\documentclass[12pt]{article}
\usepackage{setspace}
\usepackage{amsmath,amssymb,amsthm}
\usepackage{mathrsfs}
\usepackage{geometry}
\usepackage{enumitem}
\usepackage{environ}
\usepackage{setspace}
\usepackage{titlesec}
\usepackage[all,cmtip]{xy}
\usepackage{tikz}
\usetikzlibrary{matrix,arrows}
\usepackage[colorlinks=true,citecolor=blue]{hyperref}
\usepackage{mathtools}

\NewEnviron{prf}[1][]{\begin{proof}[\bf #1Proof]\BODY\end{proof}}{}
\NewEnviron{slt}[1][]{\begin{proof}[\bf #1	解]\BODY\end{proof}}{}
\newtheorem{definition}{Definition}[section]
\newtheorem{theorem}[definition]{Theorem}
\newtheorem{lemma}[definition]{Lemma}
\newtheorem{proposition}[definition]{Proposition}
\newtheorem{remark}[definition]{Remark}

\newtheorem{corollary}[definition]{Corollary}
\newtheorem{conjecture}[definition]{Conjecture}
\newtheorem{question}[definition]{Question}
\newtheorem{notation}[definition]{Notation}

\newtheorem*{DML}{Dynamical Mordell--Lang Conjecture (DML Conjecture)}

\newcommand\address[1]{#1}
\newcommand\email[1]{\emph{Email address}: #1}
  
\setlist[enumerate,1]{label=(\roman*)}
\setlist[enumerate,2]{label=(\alph*)}

\title{\textbf{Very sparse return sets in arithmetic dynamics}}
\author{She Yang}
\date{}

\begin{document}
\begin{spacing}{1.25}

\maketitle

\begin{abstract}
When studying the dynamical Mordell--Lang conjecture by using the $p$-adic method, experts have found that return sets of the $p$-adic version of the DML problem can be very sparse. In this article, we investigate the similar phenomenon for a complex analytic version of the DML problem and the DML problem for backward orbits.
\end{abstract}

\section{Introduction}

The dynamical Mordell--Lang (DML) conjecture is one of the core problems in the field of arithmetic dynamics. It was proposed by Ghioca and Tucker in \cite{GT09} and can be stated as follows.

\begin{DML}
Let $f\colon X\to X$ be an endomorphism of a quasi-projective variety over $\mathbb{C}$ and let $V$ be a closed subset of $X$. Then for every $x\in X(K)$, the return set $\{n\in\mathbb{N}\mid f^n(x)\in V\}$ is a finite union of arithmetic progressions.
\end{DML}

We write $\mathbb{N}=\mathbb{Z}_+\cup\{0\}$. An \emph{arithmetic progression} is a set of the form $\{a+bm\mid m\in\mathbb{N}\}$ where $a,b\in\mathbb{N}$. In particular, it can be a singleton.

There is an extensive literature on various cases of the DML conjecture. Two significant cases are as follows.

\begin{enumerate}
\item DML conjecture holds if $f$ is \'etale. See \cite{Bel06} and \cite[Theorem 1.3]{BGT10}.
\item DML conjecture holds if $X = \mathbb{A}^2$. See \cite{Xie17} and \cite[Theorem 4]{Xie}.
\end{enumerate}

One can consult \cite{BGT16,Xie} and the references therein for further known results.

The $p$-adic method is a very powerful tool for studying the dynamical Mordell--Lang conjecture \cite{Bel06,BGT10}. During the research, \cite[Proposition 7.1]{BGKT10} finds that there is an analytic counterexample to the $p$-adic formulation of the DML conjecture. This proposition is stated as follows.

\begin{proposition}\label{padic}
For any prime $p\geq2$, there is an increasing sequence $\{n_j\}_{j\geq1}$ of positive integers and a power series $f(z)\in\mathbb{Z}_p[[z]]$ such that $f(p^{n_j})=n_j$ for all $j$.
\end{proposition}

One can see that the growth of the sequence $\{n_j\}_{j\geq1}$ is very fast. Indeed, we have $p^{n_j}\mid n_{j+1}-n_{j}$ and hence $n_{j+1}\geq n_j+p^{n_j}$ for any $j$. This leads to the following concept of \emph{very sparse} subsets of $\mathbb{N}$. The definition is extracted from the formulation of \cite[Theorem 11.1.0.9]{BGT16}.

\begin{definition}\label{verysparse}
Let $A$ be a subset of $\mathbb{N}$.
\begin{enumerate}
\item
We say $A$ is \emph{very sparse} if $|A\cap[0,n]|=o(\log^{(m)}(n))$ for each $m\in\mathbb{Z}_+$, where $\log^{(m)}$ is the $m$-th iterated logarithm.
\item
We say $A$ is \emph{almost periodic} if it is the union of finitely many arithmetic progressions along with a very sparse set. 
\end{enumerate}
\end{definition}

\begin{remark}\label{ezrmk}
If $A$ and $B$ are very sparse (resp. almost periodic), then so does $A\cup B$, $A\cap B$, and $\{a+bn\mid n\in A\}$ for every natural numbers $a$ and $b$.
\end{remark}

As explained in \cite[p. 209]{BGT16}, Proposition \ref{padic} shows that for the endomorphism $\Phi\colon\mathbb{A}^2\to\mathbb{A}^2$ given by $(x,y)\mapsto(px,y+1)$, the set of $n\in\mathbb{Z}_+$ such that $\Phi^n(1,0)$ lies on the $p$-adic analytic curve $Y=f(X)$ is infinite but very sparse. Hence there exists no $p$-adic analytic version of the DML conjecture.

Naturally, we are also curious about whether there exists a complex analytic version of the DML conjecture. The following proposition gives a negative answer.

\begin{proposition}\label{cplx}
There exist an irrational real number $\theta$ and an analytic function $f(z)$ on the open unit disc $\mathbb{D}=\{z\in\mathbb{C}\mid|z|<1\}$ satisfy the following property. Let $a_1=\frac{1}{2}$ and $a_2=\frac{1}{2}e^{2\pi i\theta}$. Then the set $\{n\in\mathbb{Z}_+\mid f(a_1^n)=a_2^n\}$ is infinite but very sparse.
\end{proposition}

This shows that for the endomorphism $\Phi\colon\mathbb{A}^2\to\mathbb{A}^2$ given by $(x,y)\mapsto(a_1x,a_2y)$, the set of $n\in\mathbb{Z}_+$ such that $\Phi^n(1,1)$ lies on the complex analytic curve $Y=f(X)$ is not a finite union of arithmetic progressions.

Proposition \ref{cplx} leads to the following question.

\begin{question}\label{cplxq}
Let $f(z_1,\dots,z_N)$ be a germ of analytic function near $(0,\dots,0)\in\mathbb{C}^N$, i.e. it is a power series with positive convergent radius. Let $a_1,\dots,a_N$ be complex numbers such that $|a_i|<1$ for each $i$. Find $M\geq0$ such that $(a_1^n,\dots,a_N^n)$ lies in the domain of definition of $f$ for every $n\geq M$. Does the zero set $\{n\geq M\mid f(a_1^n,\dots,a_N^n)=0\}$ have to be almost periodic?
\end{question}

This problem is harder in the archimedean case than that in the non-archimedean case. Indeed, one can show that the $p$-adic version of Question \ref{cplxq} has a positive answer by using Lemma \ref{psparse}(ii). On the contrary, this complex version seems to be related to certain interesting problems in diophantine approximation. We can only give a positive answer in the 2-dimensional case.

\begin{proposition}\label{2dim}
Question \ref{cplxq} has a positive answer if $N=2$.
\end{proposition}

The similar phenomenon of very sparse return sets also occurs when considering the DML problem for backward orbits.

\begin{theorem}\label{bkwdthm}
Let $X$ be a semiabelian variety over $\mathbb{C}$ and let $f\colon X\to X$ be a surjective morphism. Let $\{b_i\}_{i\geq0}$ be a sequence of points in $X(\mathbb{C})$ satisfying $f(b_i)=b_{i-1}$ for all $i\geq1$. Let $V$ be a closed subset of $X$. Then the set $\{n\in\mathbb{N}\mid b_n\in V\}$ is almost periodic.
\end{theorem}

It turns out that this problem has a close connection with the following ``backward Skolem--Mahler--Lech" problem for torsion groups.

\begin{proposition}\label{bkwdseq}
Let $m\geq1$. Let $x:\mathbb{N}\to\mathbb{Q}/\mathbb{Z}$ be an \emph{inverse monic linear recurrence sequence} satisfying
$$
a_mx(n+m)+\cdots+a_1x(n+1)+x(n)=0,\forall n\geq0
$$
for some integers $a_1,\dots,a_m$.
Then the zero set $\{n\in\mathbb{N}\mid x(n)=0\}$ is almost periodic.
\end{proposition}

Professor Junyi Xie once made the following conjecture.

\begin{conjecture}\label{xie}(\cite[Conjecture 1.5*]{Xie18})
Let $X$ be a quasi-projective variety over $\mathbb{C}$ and let $f\colon X\to X$ be a finite morphism. Let $\{b_i\}_{i\geq0}$ be a sequence of points in $X(\mathbb{C})$ satisfying $f(b_i)=b_{i-1}$ for all $i\geq1$. Let $V$ be a closed subset of $X$. Then the set $\{n\in\mathbb{N}\mid b_n\in V\}$ is a finite union of arithmetic progressions.
\end{conjecture}

We give a counterexample of this conjecture.

\begin{proposition}\label{inverse}
There exists a sequence $x:\mathbb{N}\to\mathbb{Q}/\mathbb{Z}$ satisfying $4x(n+2)-4x(n+1)+x(n)=0$ for every $n\geq0$ for which the set $\{n\in\mathbb{N}\mid x(n)=0\}$ is infinite but very sparse. As a result, Conjecture \ref{xie} admits a counterexample for abelian surfaces.
\end{proposition}

Proposition \ref{inverse} leads to the following question.

\begin{question}\label{inverseq}
In the setting of Conjecture \ref{xie}, does the return set have to be almost periodic?
\end{question}

We make a remark about sparse return sets. \cite{BGT15} provides a sparsity result by proving that the exceptional infinite set must be of Banach density zero. But as far as we know, only analytic arguments, especially the $p$-adic argument in \cite{BGKT10}, can prove the very-sparseness of exceptional sets. See \cite[Chapter 11]{BGT16} for more informations about this topic.

The structure of this article is as follows. We deal with the complex analytic DML problem in Section \ref{sec2} and consider the DML problem for backward orbits in Section \ref{sec3}.

\textbf{Statement on A.I. use.} We learned the proofs of Propositions \ref{cplx}, \ref{2dim}, \ref{bkwdseq}, and \ref{inverse} from ChatGPT 5.6 Sol. The whole article is human-typed.

\section{Complex analytic DML}\label{sec2}

In this section, we study the complex analytic DML problem mentioned in the Introduction. Namely, we prove Propositions \ref{cplx} and \ref{2dim}.

Firstly, we aim to prove Proposition \ref{2dim}. We need the following lemma.

\begin{lemma}\label{gaplem}
Let $g\in\mathbb{C}\{z\}$ be a convergent power series satisfying $g(0)=1$. Let $x$ be a number in the open unit disc $\mathbb{D}\subseteq\mathbb{C}$. Let $\lambda,C\in\mathbb{C}^{\times}$. Find $M\geq0$ such that $x^n$ lies in the convergence domain of $f$ for every $n\geq M$. Then the set $\{n\geq M\mid\lambda^n=Cg(x^n)\}$ is almost periodic.
\end{lemma}

\begin{proof}
We may assume $x\neq0$ and $|\lambda|=|C|=1$ as other cases are easy to deal with.

We take logarithm and write $h(z)=\frac{1}{2\pi i}\log g(z)$. Then $h$ is an analytic function on a certain open disc $D(0,\varepsilon_0)$. We have $h(0)=0$ and $g(z)=e^{2\pi ih(z)}$ on that disc. Write $\lambda=e^{2\pi i\theta}$ and $C=e^{2\pi i\gamma}$ for some $\theta,\gamma\in\mathbb{R}$. We may enlarge $M$ so that $|x|^M<\varepsilon_0$ and then deduce that $\lambda^n=Cg(x^n)$ if and only if $h(x^n)-(n\theta-\gamma)\in\mathbb{Z}$ for $n\geq M$. We may assume that $h\not\equiv0$ without loss of generality as otherwise the solution set is a finite union of arithmetic progressions.

Now we choose three elements $t<t+u<t+u+v$ in $\{n\geq M\mid\lambda^n=Cg(x^n)\}$ and analyze them. Write $h(x^t)=t\theta-\gamma+m_0$, $h(x^{t+u})=(t+u)\theta-\gamma+m_1$, and $h(x^{t+u+v})=(t+u+v)\theta-\gamma+m_2$ for some integers $m_0,m_1,m_2$. We cancel $\theta,\gamma$ and get
$$
vh(x^t)-(u+v)h(x^{t+u})+uh(x^{t+u+v})\in\mathbb{Z}.
$$
We pick $a>0$ and $M_1\geq M$ such that $|h(x^s)|\leq a|x|^s$ for every $s\geq M_1$. Pick an element $c_0$ in the interval $(0,-\log|x|)$. Since $|vh(x^t)-(u+v)h(x^{t+u})+uh(x^{t+u+v})|\leq2(u+v)a|x|^t$ for every $t\geq M_1$, we can find $N_0\geq M_1$ such that the following statement holds. For every $N\geq N_0$ and every triple of solutions $t<t+u<t+u+v$ which lies in the interval $[N,e^{c_0N}]$, we have
$$
vh(x^t)-(u+v)h(x^{t+u})+uh(x^{t+u+v})=0.
$$

Let $d\geq1$ be the vanishing order of $h$ at 0, write $h(z)=c_dz^d+O(z^{d+1})$, and put $q=x^d$. Find $b>0$ and $N_1\geq N_0$ such that $|h(x^s)-c_dq^s|\leq b|q|^s|x|^s$ for every $s\geq N_1$. Find $B_1$ sufficiently large such that $|q|^u<u|q|^u<\frac{1}{4}$ for every $u\geq B_1$ and put $B_2=\lceil\frac{2B_1+1}{1-|q|}\rceil$. Pick $N_2\geq N_1$ such that $(e^{c_0}|x|)^N<\frac{|c_d|}{8b}$ for every $N\geq N_2$. Then we claim that for every $N\geq N_2$ and every triple of solutions $t<t+u<t+u+v$ which lies in the interval $[N,e^{c_0N}]$, we have
$$
u\leq B_1\ \text{and}\ v\leq B_2.
$$

Indeed, dividing the equation $vh(x^t)-(u+v)h(x^{t+u})+uh(x^{t+u+v})=0$ by $c_dq^t$, we get
$$
|v(1-q^u)-uq^u(1-q^v)|=|v-(u+v)q^u+uq^{u+v}|\leq\frac{2(u+v)}{|c_d|}b|x|^t<\frac{2b}{|c_d|}(e^{c_0}|x|)^N<\frac{1}{4}.
$$
Since $\text{LHS}\geq1-|q|^u-2u|q|^u$, we get $u\leq B_1$. Then $\text{LHS}\geq v(1-|q|)-2B_1$ and hence $v\leq B_2$.

From this claim, we deduce that there are at most $B_1+B_2+1$ solutions in any such interval $[N,e^{c_0N}]$. Therefore, the set $\{n\geq M\mid\lambda^n=Cg(x^n)\}$ is very sparse. Thus we finish the proof.
\end{proof}

\proof[Proof of Proposition \ref{2dim}]
Let $f$ be an element in the maximal ideal of $\mathbb{C}\{x,y\}$. As the ring $\mathbb{C}\{x,y\}$ is a UFD, we may assume that $f$ is an irreducible element. Then by the theory of Puiseux series (see for example \cite[Theorem 1.8.3]{CA00}), we know that locally there are two possibilities about the analytic germ defined by $f$: either it is the $y$-axis, or it can be parametrized by $(t^q,h(t))$ for some $q\in\mathbb{Z}_+$ and $h\in\mathbb{C}\{t\}$ with zero constant term. We may focus on the second case and assume $h\not\equiv0$. We may also assume $a_1a_2\neq0$.

Let $a$ be a $q$-th root of $a_1$. Take $M_0$ sufficiently large such that for every $n\geq M_0$, we have
$$
f(a_1^n,a_2^n)=0\iff a_2^n=h(\zeta a^n)
$$
for some $q$-th root of unity $\zeta$. Since the class of almost periodic sets is closed with respect of taking finite union, we may fix a $\zeta$ and analyze the set $\{n\geq M_0\mid a_2^n=h(\zeta a^n)\}$.

Write $h(\zeta z)=c_dz^dg(z)$ for some $d\in\mathbb{Z}_+$, $c_d\neq0$, and $g\in\mathbb{C}\{z\}$ which satisfies $g(0)=1$. Then the equality $a_2^n=h(\zeta a^n)$ is equivalent to
$$
\left(\frac{a_2}{a^d}\right)^n=c_dg(a^n).
$$
Then we may use Lemma \ref{gaplem} to conclude that $\{n\geq M_0\mid a_2^n=h(\zeta a^n)\}$ is almost periodic and hence finish the proof.
\endproof

Next, we prove Proposition \ref{cplx}.

\proof[Proof of Proposition \ref{cplx}]
We shall put $f(z)=ze^{2\pi iH(z)}$ where $H$ is a holomorphic function on $\mathbb{D}$ with a real Taylor expansion. We shall construct a sequence $(\theta_j)_{j\geq1}$ of rational numbers and a sequence of $\mathbb{Q}$-coefficient polynomials $(H_j)_{j\geq1}$, and then let $\theta$ and $H$ be their limit, respectively.

Start with $n_1=1$, $\theta_1=\frac{1}{2}$, and $H_1(z)=z$. We inductively construct an increasing sequence of positive integers $(n_i)_{i\geq1}$, a sequence of integers $(k_i)_{i\geq1}$, together with $(\theta_j)_{j\geq1}$ and $(H_j)_{j\geq1}$ promised above, such that $n_i\theta_j-H_j(2^{-n_i})=k_i$ holds for every $1\leq i\leq j$. Therefore, we have $k_1=0$. In the following, we denote $x_i=2^{-n_i}$ for simplicity. 

Suppose we have finished the construction for $i\leq j$. We proceed at the level $j+1$. Define auxiliary polynomial
$$
Q_j(z)=z^2\sum\limits_{i=1}^{j}\frac{n_i}{x_i^2}\prod\limits_{1\leq r\leq j,r\neq i}\frac{z-x_r}{x_i-x_r}
$$ 
which satisfies $Q_j(0)=Q_j'(0)=0$ and $Q_j(x_i)=n_i$ for $1\leq i\leq j$. Define
$$
n_{j+1}=1+\max\left\{n_j,\left\lceil2^j\left(1+\|Q_j\|_1\right)\right\rceil\right\}
$$
where $\|Q_j\|_1$ stands for the sum of the absolute values of the coefficients of $Q_j$. Set
$$
A=n_{j+1}\theta_j-H_j(x_{j+1}),\ k_{j+1}=\left\lfloor A\right\rfloor,\ D=n_{j+1}-Q_j(x_{j+1}),\ \delta_j=\frac{k_{j+1}-A}{D}
$$
and finally
$$
\theta_{j+1}=\theta_j+\delta_j,\ H_{j+1}=H_j+\delta_jQ_j.
$$
We see that everything is defined over $\mathbb{Q}$ through this inductive process. Also, we have $D>2^j+(2^j-1)\|Q_j\|_1$ and get the estimation $|\delta_j|\max\{1,\|Q_j\|_1\}<\frac{1}{2^j-1}$.

We check $n_i\theta_{j+1}-H_{j+1}(x_i)=k_i$ for every $1\leq i\leq j+1$. If $i\leq j$, then we have
\begin{align*}
n_i\theta_{j+1}-H_{j+1}(x_i)
&= n_i(\theta_j+\delta_j)-(H_j(x_i)+\delta_jQ_j(x_i))=n_i\theta_j-H_j(x_i)=k_i.
\end{align*}
Here we use $Q_j(x_i)=n_i$ and the final step is an inductive hypothesis at level $j$. For $i=j+1$, we have
\begin{align*}
n_{j+1}\theta_{j+1}-H_{j+1}(x_{j+1})
&=n_{j+1}(\theta_j+\delta_j)-(H_j(x_{j+1})+\delta_jQ_j(x_{j+1}))\\
&=A+n_{j+1}\delta_j-\delta_jQ_j(x_{j+1})=k_{j+1}.
\end{align*}
Hence we finish the check.

Now we want to put
$$
\theta=\lim\limits_{j\to\infty}\theta_j\ \text{and}\ H(z)=\lim\limits_{j\to\infty}H_j(z).
$$
We need to argue the convergence. Recall that at each step $j$, we have $\max\{|\delta_j|,\|\delta_jQ_j\|_1\}<\frac{1}{2^j-1}$. This guarantees that $\theta\in\mathbb{R}$ exists and $H(z)$ is an $\mathbb{R}$-coefficient holomorphic function on $\mathbb{D}$. Also, we see that $H(z)=z+O(z^2)$ from the construction. Furthermore, since $n_i\theta_j-H_j(2^{-n_i})=k_i$ for every $j\geq i$, we have $n_i\theta-H(2^{-n_i})=k_i$ for every $i$ by taking the limit.

We show that $\theta\notin\mathbb{Q}$. Otherwise, there exists $b\in\mathbb{Z}_+$ such that $bH(2^{-n_i})\in\mathbb{Z}$ for every $i$. But this is absurd for large $i$. Therefore $\theta\notin\mathbb{Q}$.

Now we put $f(z)=ze^{2\pi iH(z)}$ and consider the equation $f(a_1^n)=a_2^n$ where $a_1=\frac{1}{2}$ and $a_2=\frac{1}{2}e^{2\pi i\theta}$. We see that this equation is equivalent to $n\theta-H(2^{-n})\in\mathbb{Z}$ and hence the set $\{n\in\mathbb{Z}_+\mid f(a_1^n)=a_2^n\}$ contains every $n_i$. Thus it is infinite. In order to prove that it is very sparse, we only need to conclude that it does not contain any infinite arithmetic progression by Proposition \ref{2dim}.

Suppose the contrary. Then we can find $a,b\in\mathbb{Z}_+$ such that $(a+kb)\theta-H(2^{-(a+kb)})=m_k\in\mathbb{Z}$ for every $k\geq0$. By taking difference, we get $b\theta+H(2^{-(a+kb)})-H(2^{-(a+(k+1)b)})\in\mathbb{Z}$, which is absurd for large $k$ as $\theta$ is irrational. Hence we finish the proof.
\endproof

\section{DML for backward orbits}\label{sec3}

In this section, we study the DML problem for backward orbits. We prove Theorem \ref{bkwdthm} and Propositions \ref{bkwdseq} and \ref{inverse}.

Firstly, we reduce the backward DML problem for semiabelian varieties to the backward Skolem--Mahler--Lech problem.

\begin{lemma}\label{reduce}
Proposition \ref{bkwdseq} implies Theorem \ref{bkwdthm}.
\end{lemma}

The proof of Lemma \ref{reduce} will make use of the full Mordell--Lang theorem \cite{McQ95} and the Skolem--Mahler--Lech theorem \cite{Lec53}. Since we will mainly deal with linear recurrence sequences in certain abelian groups, we introduce the following notations.

\begin{notation}
Let $P(x)=\sum\limits_{i=0}^{d}b_ix^i\in\mathbb{Z}[x]$ be a polynomial with $b_db_0\neq0$. Let $G$ be an abelian group and let $a\colon\mathbb{N}\to G$ be a sequence of elements of $G$.
\begin{enumerate}
\item
The sequence $P(a)\colon\mathbb{N}\to G$ is defined by $P(a)(n)=\sum\limits_{i=0}^{d}b_ia(n+i)$.

\item
The inverse polynomial $\tilde{P}$ of $P$ is defined by $\tilde{P}(x)=\sum\limits_{i=0}^{d}b_{d-i}x^i$.

\item
We say $a$ is a \emph{linear recurrence sequence (l.r.s.)} in $G$ if $P(a)=0$ for some \emph{monic} polynomial $P$. We say $a$ is an \emph{inverse l.r.s.} if $\tilde{P}(a)=0$ for some monic polynomial $P$.
\end{enumerate}
\end{notation}

\begin{remark}\label{ezfct}
We collect some basic facts. Let $P(x),Q(x)\in\mathbb{Z}[x]$ be polynomials and let $a\colon\mathbb{N}\to G$ be a sequence.
\begin{enumerate}
\item
We have $(P+Q)(a)=P(a)+Q(a)$ and $(PQ)(a)=P(Q(a))$.

\item
Suppose $a$ is an l.r.s (resp. inverse l.r.s.). Let $k$ and $l$ be natural numbers and consider the subsequence $a'$ of $a$ defined by $a'(n)=a(k+ln)$. Then $a'$ is also an l.r.s (resp. inverse l.r.s.).
\end{enumerate}
\end{remark}

\proof[Proof of Lemma \ref{reduce}]
Firstly, notice that the backward orbit $b\colon\mathbb{N}\to X(\mathbb{C})$ in Theorem \ref{bkwdthm} is an inverse l.r.s.. Indeed, let us write the endomorphism $f$ as the composition of a self-isogeny $f_0$ with a translation. Let $P(x)\in\mathbb{Z}[x]$ be the (monic) minimal polynomial of $f_0$ and denote $Q(x)=(x-1)P(x)$. Then one can check that $\tilde{Q}(b)=0$. Therefore, we may focus on the inverse l.r.s. in $X(\mathbb{C})$. Namely, we want to prove that for every inverse l.r.s. $a\colon\mathbb{N}\to X(\mathbb{C})$ and every closed subset $V\subseteq X$, the return set $\{n\in\mathbb{N}\mid a(n)\in V\}$ is almost periodic. In the following procedure, we will often implicitly use Remark \ref{ezrmk}.

We may assume that $V$ is irreducible and $\mathrm{Im}(a)\cap V$ is dense in $V$. Notice that $\mathrm{Im}(a)$ is contained in a finite rank subgroup of $X(\mathbb{C})$. Hence we may use the full Mordell--Lang theorem \cite{McQ95} to conclude that $V$ is a translation of a semiabelian subvariety of $X$. We write $V=c+Y$ for some $c\in X(\mathbb{C})$ and semiabelian subvariety $Y\subseteq X$. Put $a'(n)=a(n)-c$ and consider the quotient map $\pi\colon X\to X/Y$. Since $a(n)\in V$ is equivalent to $\pi(a'(n))=0$ and $a'$ (and hence $\pi(a')$) is also an inverse l.r.s., we may assume $V=\{0\}$ in our task.

Next, we argue that we only need to deal with sequences of torsion points. Namely, we shall reduce to the case that $\mathrm{Im}(a)\subseteq X_{\mathrm{tor}}$. By taking Remark \ref{ezfct}(ii) into account, we only need to show that $\{n\in\mathbb{N}\mid a(n)\in X_{\mathrm{tor}}\}$ is a finite union of arithmetic progressions. Find a finite rank group $\Gamma\subseteq X(\mathbb{C})$ containing $\mathrm{Im}(a)$. Then $a(n)\in X_{\mathrm{tor}}$ is equivalent to ``$a(n)=0$ in the finite dimensional vector space $\Gamma_{\mathbb{Q}}$". By viewing $a$ as an (inverse) l.r.s. in $\Gamma_{\mathbb{Q}}$, the Skolem--Mahler--Lech theorem \cite{Lec53} helps us conclude that we can assume $\mathrm{Im}(a)\subseteq X_{\mathrm{tor}}$.

The exact sequence $0\to T\to X\to A\to0$ induces an exact sequence $0\to T_{\mathrm{tor}}\to X_{\mathrm{tor}}\to A_{\mathrm{tor}}\to0$ of torsion groups, where $T$ is a torus and $A$ is an abelian variety. Notice that both $T_{\mathrm{tor}}$ and $A_{\mathrm{tor}}$ are isomorphic to a finite direct sum of $\mathbb{Q}/\mathbb{Z}$. Denote $p$ as the projection map $X_{\mathrm{tor}}\to A_{\mathrm{tor}}$. Using Proposition \ref{bkwdseq} towards the inverse l.r.s. $p(a)$ and again using Remark \ref{ezfct}(ii), we may further assume $\mathrm{Im}(a)\subseteq T_{\mathrm{tor}}$. Using Proposition \ref{bkwdseq} again, we finish the proof.
\endproof

\begin{remark}
One can show that the inverse direction of Lemma \ref{reduce} is also valid. We will illustrate this direction in the proof of Proposition \ref{inverse}.
\end{remark}

Next, we prove Proposition \ref{bkwdseq}. The first observation is that there is a natural isomorphism $
\mathbb{Q}/\mathbb{Z}\stackrel{\sim}\longrightarrow\bigoplus\limits_{p\ \text{prime}}\mathbb{Q}_p/\mathbb{Z}_p$. Hence we can deal with every $\mathbb{Q}_p/\mathbb{Z}_p$ separately. We need the following sparsity result.

\begin{lemma}\label{psparse}
Let $K/\mathbb{Q}_p$ be a finite extension. By a \emph{convergent power series}, we mean an element in $\left\{\sum\limits_{i=0}^{\infty}a_ix^i\in\mathcal{O}_K[[x]]\ \Big|\ \lim\limits_{i\to\infty}|a_i|=0\right\}$. Let $\varpi$ be a uniformizer in $K$.
\begin{enumerate}
\item
Let $f$ be a nonzero convergent power series. Then the set $\{n\in\mathbb{N}\mid v_{\varpi}(f(n))\geq n\}$ is very sparse.

\item
Let $(g_i)_{i\geq1}$ be a sequence of nonzero convergent power series and let $(a_i)_{i\geq1}$ be a strictly increasing sequence of integers. Then $\left\{n\in\mathbb{N}\ \Big|\ \sum\limits_{i=1}^{\infty}\varpi^{a_in}g_i(n)=0\right\}$ is very sparse.
\end{enumerate}
\end{lemma}

\begin{proof}
As part (ii) is an immediate consequence of part (i), we may focus on (i).

Since $f$ is a nonzero analytic function, it has only finitely many zeros on the compact domain $\mathcal{O}_K$. Denote the zeros as $t_1,\dots,t_r$ and denote $d_1,\dots,d_r$ as their multiplicities. There are nonnegative integers $e_1,\dots,e_r$ and $c_1,\dots,c_r$ such that for every $1\leq i\leq r$, we have $v_{\varpi}(f(n))=d_iv_{\varpi}(n-t_i)+c_i$ whenever $v_{\varpi}(n-t_i)\geq e_i$. Using compactness again, we see there exists $M\geq0$ such that $\{n\geq M\mid v_{\varpi}(f(n))\geq n\}\subseteq\bigcup\limits_{i=1}^{r}\overline{D}(t_i,|\varpi|^{e_i})$ where $\overline{D}(t_i,|\varpi|^{e_i})\coloneqq\{x\in\mathcal{O}_K\mid|x-t_i|\leq|\varpi|^{e_i}\}$. We only need to prove that each $\{n\geq M\mid v_{\varpi}(f(n))\geq n\}\cap\overline{D}(t_i,|\varpi|^{e_i})$ is very sparse.

For every $i$, this set equals to $\left\{n\geq M\ \big|\ v_{\varpi}(n-t_i)\geq\max\{e_i,\frac{n-c_i}{d_i}\}\right\}$. For any two consecutive elements $n<n'$ in this set, we have $v_p(n'-n)v_{\varpi}(p)=v_{\varpi}(n'-n)\geq\frac{n-c_i}{d_i}$. Hence $n'-n\geq p^{\frac{n-c_i}{d_iv_{\varpi}(p)}}$. This gives the very-sparsity and thus we finish the proof.
\end{proof}

\begin{remark}
Let us consider the analogue of Lemma \ref{psparse}(ii) in the complex setting. As suggested in Lemma \ref{gaplem}, we concern about sets with the shape $\{n\in\mathbb{N}\mid|e^{2\pi in\theta}-1|\leq\varepsilon^n\}$ for some $\theta\in\mathbb{R}$ and $\varepsilon>0$. This condition amounts to requiring that $\|n\theta\|<\varepsilon^n$, where $\|n\theta\|$ stands for the distance from $n\theta$ to $\mathbb{Z}$. By considering pathological Liouville numbers, we find that we can gain nothing stronger than the zero density of such sets.
\end{remark}

\begin{corollary}\label{cor}
Let $x\colon\mathbb{N}\to\mathbb{Q}_p$ be an (inverse) l.r.s.. Then $\{n\in\mathbb{N}\mid x(n)\in\mathbb{Z}_p\}$ is almost periodic.
\end{corollary}

\begin{proof}
Indeed, the conclusion holds for every linear recurrence sequence in $\mathbb{Q}_p$ with coefficients also in $\mathbb{Q}_p$ (instead of $\mathbb{Z}$). We can find a finite extension field $K$ of $\mathbb{Q}_p$ and write the general term formula $x(n)=\sum\limits_{i=1}^{r}P_i(n)\lambda_i^n$ with polynomials $P_i\in K[x]$ and characteristic roots $\lambda_i\in K^{\times}$. Let $\varpi$ be a uniformizer of $K$ and write $\lambda_i=\varpi^{a_i}\alpha_i$ for each $i$, where $a_i\in\mathbb{Z}$ and $\alpha_i\in\mathcal{O}_K^{\times}$. We can find $D\in\mathbb{Z}_+$ such that each exponential function $\alpha_i^{Dn}$ can be represented as a convergent power series (in the sense of Lemma \ref{psparse}) with variable $n$. Passing to arithmetic progressions $\{a+Dn\mid n\in\mathbb{N}\}$ for $0\leq a<D$, we only need to prove the following statement. For integers $b_1<\cdots<b_s$ and nonzero convergent power series $f_1,\dots,f_s$, the set
$$
\left\{n\in\mathbb{N}\ \Bigg|\ v_{\varpi}\left(\sum\limits_{i=1}^{s}\varpi^{b_in}f_i(n)\right)\geq c\right\}
$$
is almost periodic for every $c\in\mathbb{N}$.

We may assume $b_1\leq0$ since otherwise the conclusion is easy. If $b_1<0$, then this set is very sparse according to Lemma \ref{psparse}(i). If $b_1=0$, then for any sufficiently large integer $n$ the condition is equivalent to $v_{\varpi}(f_1(n))\geq c$. By picking $M\geq\frac{c}{v_{\varpi}(p)}$ and considering $n$ modulo $p^M$, we find that this set is a finite union of arithmetic progressions. Thus we finish the proof.
\end{proof}

Now we prove Proposition \ref{bkwdseq}. The key point is that an inverse l.r.s. in $\mathbb{Q}_p/\mathbb{Z}_p$ can be lifted to an inverse l.r.s. in $\mathbb{Q}_p$ which satisfies the same relation. This needs the compactness of $\mathbb{Z}_p$ and we cannot do this towards sequences in $\mathbb{Q}/\mathbb{Z}$.

\proof[Proof of Proposition \ref{bkwdseq}]
As we have mentioned before, the isomorphism $\mathbb{Q}/\mathbb{Z}\stackrel{\sim}\longrightarrow\bigoplus\limits_{p\ \text{prime}}\mathbb{Q}_p/\mathbb{Z}_p$ guarantees that we can deal with every component $\mathbb{Q}_p/\mathbb{Z}_p$ separately. Therefore, we only need to prove that the zero set of every inverse l.r.s. in $\mathbb{Q}_p/\mathbb{Z}_p$ is almost periodic.

Let $x\colon\mathbb{N}\to\mathbb{Q}_p/\mathbb{Z}_p$ be an inverse l.r.s. satisfying $x(n)+\sum\limits_{i=1}^{d}a_ix(n+i)=0$ for every $n\in\mathbb{N}$. Here the coefficients $a_1,\dots,a_d$ are integers and $a_d\neq0$. We arbitrarily fix an lifting $\tilde{x}\colon\mathbb{N}\to\mathbb{Q}_p$ of $x$. Write 
\[
A = 
\begin{bmatrix}
-a_1 & -a_2 & -a_3 & \cdots & -a_d \\
1 & 0 & 0 & \cdots & 0 \\
\vdots & \ddots & \ddots & \ddots & \vdots \\
0 & \cdots & 1 & 0 & 0 \\
0 & \cdots & 0 & 1 & 0
\end{bmatrix}_{d\times d}
\]
and $y_n=(\tilde{x}(n),\dots,\tilde{x}(n+d-1))^{\top}$. Then $A$ is a linear automorphism of $\mathbb{Q}_p^d$ and we have $Ay_{n+1}-y_n\in\mathbb{Z}_p^d$ for every $n$. Also, we have $A^{n+1}(\mathbb{Z}_p^d)\subseteq A^n(\mathbb{Z}_p^d)$ for every $n$. Hence we may denote $V_n$ as the coset $A^ny_n+A^n(\mathbb{Z}_p^d)$ and see that $V_{n+1}\subseteq V_n$ for every $n$. So $(V_n)_{n\geq0}$ is a decreasing sequence of nonempty compact sets and we conclude that $\bigcap\limits_{n=0}^{\infty}V_n\neq\emptyset$.

Pick $v\in\bigcap\limits_{n=0}^{\infty}V_n$ and denote $v_n=A^{-n}v$. Then we have $Av_{n+1}=v_n$ and $v_n-y_n\in\mathbb{Z}_p^d$ for every $n$. Also, there exists a sequence $w\colon\mathbb{N}\to\mathbb{Q}_p$ satisfying the relation $w(n)+\sum\limits_{i=1}^{d}a_iw(n+i)=0$ such that $v_n=(w(n),\dots,w(n+d-1))^{\top}$ for every $n$. Thus we have
$$
x(n)=0\iff\tilde{x}(n)\in\mathbb{Z}_p\iff w(n)\in\mathbb{Z}_p
$$
and hence $\{n\in\mathbb{N}\mid x(n)=0\}$ is almost periodic according to Corollary \ref{cor}. Using Remark \ref{ezrmk}, we finish the proof by assembling those finitely many involved $p$-components.
\endproof

Theorem \ref{bkwdthm} then follows from Proposition \ref{bkwdseq} and Lemma \ref{reduce}.

Finally, we prove Proposition \ref{inverse}.

\proof[Proof of Proposition \ref{inverse}]
Firstly, we construct this inverse l.r.s. satisfying
$$
4x(n+2)-4x(n+1)+x(n)=0.
$$

Consider the sequence of integers $(n_i)_{i\geq0}$ defined by $n_0=0$ and $n_{i+1}=n_i+2^{n_i+1}$ for every $i$. For every positive integer $k$, there exists a unique $r_k\in\{0,\dots,2^k-1\}$ such that $2^k\mid n_i-r_k$ for every $i$ satisfying $n_i+1\geq k$. This implies $2^k\mid r_{k+1}-r_k$ for every $k$.

We define $x(n)$ as the class of $\frac{n-r_{n+1}}{2^{n+1}}\in\mathbb{Q}$ in $\mathbb{Q}/\mathbb{Z}$. We need to check $4x(n+2)-4x(n+1)+x(n)=0$, which is equivalent to $2^{n+1}\mid((n+2)-r_{n+3})-2((n+1)-r_{n+2})+(n-r_{n+1})$. This is valid as RHS equals to $(r_{n+2}-r_{n+1})-(r_{n+3}-r_{n+2})$.

It remains to determine the zero set. We have $x(n)=0\iff 2^{n+1}\mid n-r_{n+1}$. Thus we tautologically have $x(n_i)=0$ for every $i$. We claim that $\{n\in\mathbb{N}\mid x(n)=0\}=\{n_i\mid i\geq0\}$. Otherwise, there exists an element $n$ in the zero set which lies between $n_i$ and $n_{i+1}$ for some $i$, i.e. we have $n_i<n<n_{i+1}$. But then the four numbers $n,r_{n+1},r_{n_i+1},n_i$ are congruent modulo $2^{n_i+1}$, which implies that $n\geq n_i+2^{n_i+1}=n_{i+1}$. This contradiction finishes the proof of the claim. Hence the zero set is infinite but very sparse.

Finally, we use this inverse l.r.s. to construct a counterexample of Conjecture \ref{xie}. Let $E$ be an elliptic curve and let $A=E\times E$. Let $f$ be the self-isogeny of $A$ defined by $(a,b)\mapsto(b,4b-4a)$ for $a,b\in E(\mathbb{C})$. Since the torsion group $E_{\mathrm{tor}}$ is isomorphic to $(\mathbb{Q}/\mathbb{Z})^2$, we may fix an injection $\mathbb{Q}/\mathbb{Z}\hookrightarrow E_{\mathrm{tor}}$ and regard the sequence $x$ constructed above as a sequence of torsion points of $E$. Define $\alpha_n=(x(n+1),x(n))\in A_{\mathrm{tor}}$. Then the sequence $(\alpha_n)_{n\geq0}$ is a backward orbit of the endomorphism $f$ and the return set $\{n\in\mathbb{N}\mid\alpha_n\in E\times\{0\}\}=\{n\in\mathbb{N}\mid x(n)=0\}$ is infinite but very sparse. Hence Conjecture \ref{xie} fails for abelian surfaces.
\endproof

\section*{Acknowledgements}
We are grateful to Junyi Xie, the advisor of the author, for suggesting the topic of this article.

This work is supported by the National Natural Science Foundation of China Grant No. 12271007.

\bibliographystyle{alpha}
\bibliography{reference}

\address{Beijing International Center for Mathematical Research, Peking University, Beijing 100871, China}

\email{ys-yx@pku.edu.cn}

\end{spacing}
\end{document}